\documentclass{amsart}
\usepackage{amsmath, amsfonts, amssymb, amsthm}

\newtheorem{theorem}{Theorem}
\newtheorem{lemma}[theorem]{Lemma}
\newtheorem{cor}[theorem]{Corollary}
\newtheorem{prop}[theorem]{Proposition}

\theoremstyle{definition}

\def\min{\mathop{\mathrm{min}}}

\def\RR{\mathbb R}

\def\PP{\mathbb P}
\def\K{{\bf K}}

\def\cal{\mathcal }
 \def\ord{\text{ord}}
 
\def\p{\mathbf p}

\def\dkh{the assumption ($\mathcal{H}$)}

\def\noi{{\noindent}}

\def\gen{\mathfrak g}

\usepackage{ulem}

\begin{document}
\title[Heights of Function Field Points]{Heights of Function Field Points on Curves Given by Equations
with Separated Variables}
\author{Ta Thi Hoai An}
\address{Institute of Mathematics, 18 Hoang Quoc Viet, Cau Giay\\ 10307   Hanoi  \\
Viet Nam}
\email{tthan@math.ac.vn}
\author{Nguyen Thi Ngoc Diep}
\address{Department of Mathematics\\Vinh University\\Vietnam} \email{ngocdiepdhv@gmail.com}
\begin{abstract} Let $P$ and $Q$ be polynomials
in one variable over an algebraically closed field $k$ of
characteristic zero. Let $f$ and $g$ be elements of a function field
$\K$ over $k$ such that $P(f)=Q(g).$ We give conditions on $P$ and $Q$
such that the height of $f$ and $g$ can be effectively bounded, and moreover,
we give sufficient conditions on $P$ and $Q$ under which $f$ and $g$ must
be constant.
\end{abstract}

\thanks{Financial support provided to the first named author
by Alexander von Humboldt Foundation and  ICTP and to  the both authors
by  Vietnam's National Foundation for Science and Technology Development (NAFOSTED)}
\thanks{2000\ {\it Mathematics Subject Classification.} Primary 14H05
Secondary 30D35 14H55.}

\maketitle


\section{Introduction }

Let $k$ be an algebraically closed field of characteristic zero, and
let $P$ and $Q$ be polynomials in $k[X].$ Determining when
algebraic curves defined by equations of the form $P(X)-Q(Y)=0$
have irreducible components of geometric genus zero or one is or interest
both in arithmetic and complex function theory because it is related
to when such curves may have infinitely many rational solutions in a number
field or non-constant meromorphic function solutions. 
Genus zero components are also related to the study of what
are known as ``uniqueness polynomials.'' Special cases
of this problem have been considered, but to date there has been no
complete characterization of when an algebraic curve of the
form $P(x)-Q(y)=0$ has no irreducible components of genus at most one. For example,  in \cite{AE},  \cite{AWW} and  \cite{AWW2},  authors considered the problem for cases of genus zero over non-archimedean field. 
 Fujimoto in \cite{Fuj} gave some  sufficient conditions for the problem over complex number field under  assumption $Q=cP$ for some constant $c$  and  the polynomial   $P$ satisfies \textit{Hypothesis I,} that  is  $P$
is injective on the set of distinct zeros of $P'.$

Let  $C$ be a smooth curve of genus $\gen$ over $k,$
and let $\K$ be its function
field. Let
$F_1(X,Y,Z)$ be the homogenization of $[P(X)-P(Y)]/[X-Y]$ and
$F_c(X,Y,Z)$, $c\ne 0,1$, be the homogenization of
$P(X)-cP(Y)$. If $f$ and $g$ are algebraic functions in  $\K$
such that $P(f)=cP(g)$ for some nonzero constant $c$,
then the morphism
$\Phi:=(f,g,1)$ is a morphism from $C$ into  the curve $[F_c(X,Y,Z)=0]$.
By the Hurwitz theorem this cannot happen if the curves $[F_c(X,Y,Z)=0]$
have no components of genus $\le \gen$.  
In \cite{AW1}, An and Wang gave sufficient conditions on a polynomial
$P$ satisfying Hypothesis~I such that the height of any solution
$(f,g)$ with $f,g\in\K$ such that $P(f)=cP(g)$ could be effectively
bounded above.
The purpose of this paper is to consider more general separated
variable equations of the form $P(x)-Q(y)$ and to give
some conditions on the polynomials $P$ and $Q$ such that if $f$ and $g$ are
elements of $\K$ satisfying the equation $P(f)=Q(g),$ then the heights of $f$
and $g$ can be effectively bounded.

For each point $\p\in C$, we may choose a uniformizer $t_\p$ to define a
normalized order function $$v_\p:= \ord_\p:\K\to \RR\cup\{\infty\}$$ at
 $\p$. For a non-zero element $f\in \K$, the height $h(f)$ counts
the number of poles of $f$ with multiplicities, i.e.
 $$h(f):=\displaystyle\sum_{\p\in C} -\min\{0,\,v_\p(f)\}.$$

 For $[f,g]\in\PP^1(\K)$,
 its height is defined by $$h(f,g) := \displaystyle\sum_{\p\in C} -\min\{v_\p(f),v_\p(g)\}.$$
 Clearly, $h(f) = h(f,1)$.

 From now we will let $P(X)$ and $Q(X)$ be nonconstant polynomials of degree
 $n$ and $m$, respectively, in $k[X]$. Without loss of generality,
 throughout the paper we will assume that $n\geq m$. We will  denote  by $\alpha_1, \alpha_2, ..., \alpha_l$ and
 $\beta_1, \beta_2, ..., \beta_h$ the distinct roots of $P'(X)$ and $Q'(X)$, respectively. We will use $p_1, p_2, ..., p_l$
 and $q_1, q_2, ..., q_h$ to denote the multiplicities of the roots in
 $P'(X)$ and $Q'(X)$, respectively. Thus, for some $a$, $b$ in $k$,
 $$P'(X)=a{(X-\alpha_1)}^{p_1}{(X-\alpha_2)}^{p_2}...{(X-\alpha_l)}^{p_l}$$
 $$Q'(X)=b{(X-\beta_1)}^{q_1}{(X-\beta_2)}^{q_2}...{(X-\beta_h)}^{q_h}.$$

 Recall, the polynomial $P(X)$  satisfies \textit{Hypothesis I} if
 $$P(\alpha_i) \neq P(\alpha_j)\ {\rm whenever}\ i\neq j,$$
 or in other words $P$ is injective on the roots of $P'$.

If one of the polynomials $P$ or $Q$ is linear, say $P(X)=aX+b,$  then  $( \frac 1a Q(f)-b,f)$ is a solution of the equation $P(X)=Q(Y),$ where   $f$ is some nonconstant element of $\K$. Hence, from now we always assume that both $P$ and $Q$ are not linear polynomials.
 The main results are as follows.
 \begin{theorem}\label{th1} Suppose that $f$ and $g$ are two distinct
 non-constant rational functions in $\K$ such that $P(f) = Q(g)$. Let
 \begin{align*}&B_0=\{i\mid 1\leq i\leq l, \ P(\alpha_i)\ne
Q(\beta_j)\text{   for all $ j=1,...,h\}$\quad and }\\
&B_1=\{i\mid 1\leq i\leq h, \ Q(\beta_i)\ne P(\alpha_j) \text{  for
all  $j=1,...,l\}$}.\end{align*} 
Then
\begin{enumerate}\item[(a)]$n h(f)={m}h(g);$
\item[(b)]$\Big({\displaystyle\sum_{i\in B_0}p_i}
-\frac{m+n}{m}\Big) h(f)\leq{2\gen-2}; $
\item[(c)]$\Big(\displaystyle\sum_{i\in B_1}q_i  -\frac{2m}{n}\Big) \ h(g)\leq 2\gen-2.$
\end{enumerate}
\end{theorem}
As a consequence of Theorem \ref{th1} we have the following corollary.

\begin{cor}With the same hypothesis as in Theorem~\ref{th1}.
\begin{enumerate}\item[(a)] If either ${\displaystyle\sum_{i\in B_0}p_i}
-\frac{m+n}{m}>0$ or $\displaystyle\sum_{i\in B_1}q_i  -\frac{2m}{n}>0$,
then the heights of $f$ and $g$ are effectively bounded above;
\item[(b)] If either ${\displaystyle\sum_{i\in B_0}p_i}
-\frac{m+n}{m}>\max\{0,2\gen-2\}$ or $\displaystyle\sum_{i\in B_1}q_i  -\frac{2m}{n}>\max\{0,2\gen-2\}$, then
$f$ and $g$ are constant.
\end{enumerate}
\end{cor}

In order to state  Theorem  \ref{th2} clearly,
we need to introduce the following notation:

\medskip

\noi {\bf Notation.} \quad We put:
\begin{align*}&A_0:=\{(i, j)\mid 1\leq i\leq l, 1\leq j\leq h, \
P(\alpha_i)=Q(\beta_{j})\},\\
&A_1:=\{(i, j)\mid (i, j)\in A_0,\  p_i>q_j\},\\
&A_2:=\{(i, j)\mid (i, j)\in A_0 ,\
p_i<q_j\}.
\end{align*}

\begin{theorem}\label{th2} Let $P(X)$ and $Q(X)$ satisfy Hypothesis I. Suppose that $f$ and $g$ are two distinct
 non-constant rational functions in $\K$ such that $P(f) = Q(g)$. Then
$\Big(\Big(\displaystyle\sum_{(i, j)\in A_1}p_i-\frac{n}{m}q_j\Big)+
\displaystyle\Big(\sum_{1\leq i\leq l,(i, j)\notin A_0}p_i\Big)-\frac{m+n}{m}\Big) \ h(f)\leq
2\gen-2.$
\end{theorem}

\begin{cor}\label{th4}  Let $P(X)$ and $Q(X)$ satisfy Hypothesis I.  Suppose that $f$ and $g$ are two
 rational functions in $\K$ such that $P(f) = Q(g)$. If 
$$  \displaystyle\sum_{(i, j)\in A_1}(p_i-\frac{n}{m}q_j)+
\displaystyle\sum_{1\leq i\leq l,(i, j)\notin A_0}p_i-\frac{m+n}{m}> \max\{0,2\gen-2\}.$$
Then $f$ and $g$ are constants:
\end{cor}

We will see that Hypothesis I tells us
that there are not too many $(i,j)\in A_0$ (see Lemma~\ref{lm8}).

Especially, when the genus $\gen =0$ and  the degrees of $P(X)$ and $Q(X)$ are the same, the
following theorem gives a necessary and sufficient condition such that
the equation $P(X)=Q(Y)$ has no non-constant rational function solution.

\begin{theorem}\label{th3}
Let $\gen=0$ and $P(X),\  Q(X)$  satisfy Hypothesis  I and suppose $n=m$. Suppose that $f$ and $g$ are two
 rational functions in $\K$ such that $P(f) = Q(g)$. Then $f$ and $g$ are constant  if and only if $P(X)$ and
$Q(X)$ satisfy none of  the following conditions
 \begin{enumerate}
\item[\bf (A)]\ $P(X)-Q(Y)$ has a linear factor.
  \item[\bf (B)]\ $l=1$, $h=2$, $p_1=q_1+1$, $q_2=1$ and
$P(\alpha_1)=Q(\beta_1)$; or $h=1$, $l=2$, $q_1=p_1+1$, $p_2=1$ and
$P(\alpha_1)=Q(\beta_1)$.
  \item[\bf (C)]\ $l=h=2$, $p_2=q_2=1$, $p_1=q_1$ and $P(\alpha_1)=Q(\beta_1)$.
  \item[\bf (D)]\ $l=h=3$, $p_i=q_i=1$ for any $i=1,2,3$ and
$P(\alpha_i)=Q(\beta_i)$ for any
  $i=1,2,3$ (after changing the indices).
    \item[\bf (E)]\  $l=h=p_1=q_1=1.$
\end{enumerate}
\end{theorem}

\medskip

{\it Acknowledgments.} 
We would like to thank the referee for her/his very careful reading and helpful suggestions.

\section{Some lemmas}
For simplicity of notation, for $i\ge 1$, $t\in \K\setminus  k$ and
$\eta\in
\K$, we denote by
$$  d_t^i\eta:=\frac{d^i\eta}{dt^i}, \qquad
d_\p^i\eta:=\frac{d^i\eta}{dt_\p^i}.
$$
We first recall the following well-known properties, which follow  from the
Riemann-Roch theorem and the sum formula.

\begin{prop}\label{RR} Let  $\eta\ne 0\in {\bf K}$ and
$[f,g]\in\PP^1(\K)$.  We have
\begin{enumerate}
\item[(i)]$\displaystyle\sum_{\p\in C} v_\p(d_\p \eta)=2\gen -2$ if $\eta$ is not constant.
\item[(ii)] $\displaystyle\sum_{\p\in C} v_\p( \eta)=0$.
\item[(iii)] $h(\eta f,\eta g )=h(f,g).$
\end{enumerate}
\end{prop}

In order to study some sufficient conditions ensuring that the equation $P(X)=Q(Y)$
has no non-constant rational function solutions, the basic idea is as follows.
Suppose there are two distinct non-constant rational functions $f$ and $g$
in $\K$ such that $P(f)=Q(g)$. We will study the height of $f$ and $g$
and give upper bounds for $h(f)$ and $h(g)$. We first give an upper bound
for $h(P'(f),Q'(g))$ thanks to the following lemma.

\begin{lemma}\label{lmCT} Suppose that   $f$ and
$g$ are distinct non-constant rational functions in $\K$ such that $P(f)=Q(g)$.
Then
\begin{enumerate}
\item[(i)] $nh(f)=mh(g);$
\item[(ii)] $ h({P'(f)},{Q'(g)})+
\displaystyle\sum_{\p\in C}  \min\{v_{\p}^0(d_\p f), v_{\p}^0(d_\p g)\}
\leq \frac{m+n}{m}h(f)+2\gen-2; $
\end{enumerate}
where $v_{\p}^0(\eta) := \max  \{0, v_{\p}(\eta)\}$  for $\eta \in \K^*.$
\end{lemma}

\begin{proof}  Since $P(f)=Q(g)$,  for any $\p\in C$ such that $v_{\p}(f)<0$, we have
 $$n v_\p(f)=v_\p(P (f))=v_\p(Q (g))=m v_\p(g).$$
 Hence, $n h(f)=m h(g)$. This also yields that
$$ d_tfP'(f)=d_tgQ'(g), $$
  for  $t$ in $\K\setminus{\bf k}$, and hence
$$h(P'(f),Q'(g))=h(\frac{P'(f)}{Q'(g)},1)=h(\frac{d_tg}{d_tf},1)=h(d_tf,d_tg).$$
 Since
$d_{\p}f=d_tfd_{\p}t,$ it follows that $v_{\p}(d_{\p}f)=v_{\p}(d_tf)+v_{\p}(d_{\p}t),$
and hence
\begin{align*}  v_{\p}(d_tf)=v_{\p}(d_{\p}f)-v_{\p}(d_{\p}t).
\end{align*}
We have
\begin{align*} & h({P'(f)},{Q'(g)})= h(d_t f,d_tg)= \displaystyle\sum_{\p\in C}-\min\{v_{\p}(d_tf),v_{\p}(d_tg)\}\cr
 &=\displaystyle\sum_{\p\in C}-\min\{v_{\p}(d_{\p}f)-v_{\p}(d_{\p}t),v_{\p}(d_{\p}g)-v_{\p}(d_{\p}t)\}\cr
 &=\displaystyle\sum_{\p\in C}v_{\p}(d_{\p}t)+\displaystyle\sum_{\p\in C}-\min\{v_{\p}(d_{\p}f),v_{\p}(d_{\p}g)\}\cr
 &= \displaystyle\sum_{\p\in C}v_{\p}(d_{\p}t)+\displaystyle\sum_{v_{\p}(f)< 0}-\min\{v_{\p}(d_{\p}f),
v_{\p}(d_{\p}g)\}-\displaystyle\sum_{v_{\p}(f)\ge0}\min\{v_{\p}(d_{\p}f),
v_{\p}(d_{\p}g)\}.
\end{align*}
If $v_\p(f)<0$, then $v_{\p}(d_{\p}f)=v_\p(f)-1$ and
$v_\p(d_{\p}g)=v_{\p}(g)-1=\frac{n}{m}v_{\p}(f)-1.$ \\
If $v_\p(f)\ge 0$, then $v_{\p}(d_{\p}f)\ge 0$ hence
$v_{\p}^0(d_{\p}f)=\max\{0,v_{\p}(d_{\p}f)\}=v_{\p}(d_{\p}f)$
and $v_{\p}^0(d_{\p}g)=v_{\p}(d_{\p}g).$

All together, we have
\begin{align*}&h({P'(f)},{Q'(g)})\cr
&=2\gen-2+
\displaystyle\sum_{v_{\p}(f)< 0}-\min\{v_\p(f)-1,\frac{n}{m}v_{\p}(f)-1\}
-\displaystyle\sum_{v_{\p}(f)\ge0}\min\{v_{\p}^0(d_{\p}f),
v_{\p}^0(d_{\p}g)\}\cr
&= 2\gen-2+\displaystyle\sum_{v_{\p}(f)< 0}(-\frac{n}{m}v_{\p}f+1)  -\displaystyle\sum_{\p\in C}
\min\{v_{\p}^0(d_{\p}f), v_{\p}^0(d_{\p}g)\} \cr
  &= 2\gen-2 + \frac{n}{m}h(f)+\#\{\p\in C\mid v_{\p}(f)< 0\} -\displaystyle\sum_{\p\in C}
\min\{v_{\p}^0(d_{\p}f), v_{\p}^0(d_{\p}g)\}.
\end{align*}
 Clearly, $\#\{\p\in C\mid v_{\p}(f)< 0\}\le h(f)$. Therefore,
 \begin{align*}
 h({P'(f)},{Q'(g)})\leq 2\gen-2 + \frac{m+n}{m}h(f) -\displaystyle\sum_{\p\in C}
\min\{v_{\p}^0(d_{\p}f), v_{\p}^0(d_{\p}g)\}.
\end{align*}
Hence
\begin{align*}
h({P'(f)},{Q'(g)})+\displaystyle\sum_{\p\in C}  \min\{v_{\p}^0(d_\p f), v_{\p}^0(d_\p g)\}
\leq \frac{m+n}{m}h(f)+2\gen-2.
\end{align*}
\end{proof}

\noindent{\bf Remark.}
We note that Lemma \ref{lmCT} gives an upper bound for $h(P'(f),Q'(g)):$
\begin{equation}\label{bdtCT}
h({P'(f)},{Q'(g)})  \leq \frac{m+n}{m}h(f)+2\gen-2
\end{equation}
since $  v_{\p}^0(d_\p f)\ge 0$ and $v_{\p}^0(d_\p g) \ge 0$.

\begin{lemma}\label{hG} Suppose that   $f$ and
$g$ are distinct non-constant rational functions in $\K$ such that $P(f)=Q(g)$.
Then
$$ -\displaystyle\sum_{\p\in C, v_\p(f)<0}\min\{v_\p({P'(f)}),v_\p({Q'(g)})\}=(n-1)h(f)=(n-1)\frac mn h(g).$$
\end{lemma}

\begin{proof}
Since $P(f)=Q(g)$ it follows that at a point $\p\in C$ such that $v_\p(f)<0$,  $v_\p(P'(f))=(n-1)v_\p(f), $ $v_\p(g)<0$ and $nv_\p(f)=mv_\p(g)$. We also have
$v_\p(Q'(g))=(m-1)v_\p(g)$.
Therefore,
\begin{align*}\min\{v_\p({P'(f)}),v_\p({Q'(g)})\}&=\min\{(n-1)v_\p(f),(m-1)\frac{n}{m}v_\p(f)\}\cr
&=(n-1)v_\p(f)=(n-1)\frac mnv_\p(g),
\end{align*}
which yields
\begin{align*}-\displaystyle\sum_{\p\in C, v_\p(f)<0}\min\{v_\p({P'(f)}),v_\p({Q'(g)})\}&=(n-1)h(f) =(n-1)\frac mn h(g).
\end{align*}
\end{proof}

We now turn to finding a lower bound for $h(P'(f),Q'(g))$ in terms of
$h(f)$. To find a lower bound for $h(P'(f),Q'(g))$, we will need
to find  an element $G$ in $\K$ such that: firstly the height of $G$ is not too big
and, secondly the vanishing order of $G$ at each point of the curve is at least
the minimum of the vanishing order of 
$P'(f)$ and $Q'(g)$.

\section{Proof of Theorem  \ref{th1}}
Recall that we have set:
\begin{align*}&B_0=\{i\mid 1\leq i\leq l,\  P(\alpha_i)\ne
Q(\beta_j), \text{ for all $ j=1,...,h\},$}\\
&B_1=\{i\mid 1\leq i\leq h,\  Q(\beta_i)\ne P(\alpha_j) , \text{  for
all  $j=1,...,l\}$.} \end{align*}
By rearranging the indices if necessary, we may assume
$$B_0=\{1, 2, \ldots, l_1\} \,\,{\rm and}\,\, B_1=\{1, 2, \ldots, h_1\}.$$

\begin{proof}[\textbf{Proof of Theorem \ref{th1}.}] The assertion (a) is given in Lemma~\ref{lmCT}~(i).

We thus prove (b) and (c), beginning with (b).
We take
$$G(f):=\displaystyle\prod_{i\notin B_0}{(f-\alpha_i)}^{p_i}=\displaystyle\prod_{i=l_1+1}^l{(f-\alpha_i)}^{p_i},$$
and $d:=\displaystyle\sum_{i=l_1+1}^l p_i$. We have
\begin{align}\label{hPQphay}h(P'(f),Q'(g)) &= h\left(\frac{P'(f)}{G(f)},\frac{Q'(g)}{G(f)}\right)\cr
&= -\displaystyle\sum_{\p\in C}\min\{v_\p(\frac{P'(f)}{G(f)}),v_\p(\frac{Q'(g)}{G(f)})\}\cr
&=-\displaystyle\sum_{v_\p(f)<0}\min\{v_\p(\frac{P'(f)}{G(f)}),v_\p(\frac{Q'(g)}{G(f)})\}\cr
&\quad\quad-\displaystyle\sum_{v_\p(f)\geq0}\min\{v_\p(\frac{P'(f)}{G(f)}),v_\p(\frac{Q'(g)}{G(f)})\}.
\end{align}
If $v_\p(f)<0$ then $v_\p(G(f))=dv_\p(f)$.
Therefore, by Lemma~\ref{hG},
\begin{align}\label{v}-\displaystyle\sum_{v_\p(f)<0}\min\{v_\p(\frac{P'(f)}{G(f)}),v_\p(\frac{Q'(g)}{G(f)})\}
&=-\displaystyle\sum_{v_\p(f)<0}\min\{v_\p({P'(f)}),v_\p({Q'(g)})\}\cr
&\quad\quad+\displaystyle\sum_{v_\p(f)<0}v_\p(G)\cr
&=(n-1-d)h(f).
\end{align}
Putting (\ref{hPQphay}) and (\ref{v}) together, we have
\begin{align}\label{h}h(P'(f),Q'(g))
&=(n-1-d)h(f)-\displaystyle\sum_{v_\p(f)\geq0}\min\{v_\p(\frac{P'(f)}{G(f)}),v_\p(\frac{Q'(g)}{G(f)})\}.
\end{align}
 Therefore,  to obtain a lower bound on $h(P'(f),Q'(g)),$  our goal
is  to prove
\begin{equation}\label{*}\displaystyle\sum_{v_\p(f)\geq0}\min\{v_\p(\frac{P'(f)}{G(f)}),v_\p(\frac{Q'(g)}{G(f)})\}\le 0.
\end{equation}

We have
$$\frac{P'(f)}{G(f)}=\frac{\displaystyle\prod_{i=1}^{l}{(f-\alpha_i)}^{p_i}}
{\displaystyle\prod_{i=l_1+1}^{l}{(f-\alpha_i)}^{p_i}}=\displaystyle\prod_{i=1}^{l_l}{(f-\alpha_i)}^{p_i}.$$
For our purpose, taking into account displayed formula (\ref{*}),  we only have to consider points $\p\in C$ such that $v_\p(f)\geq0$. We first consider those $\p$ satisfying $\frac{P'(f)}{G(f)}(\p)\ne0$, (i.e $v_\p(\frac{P'(f)}{G(f)})=0)$; hence
\begin{align}\label{a1}\min\{v_\p(\frac{P'(f)}{G(f)}),v_\p(\frac{Q'(g)}{G(f)})\}\le 0.\end{align}
At such point satisfying $\frac{P'(f)}{G(f)}(\p)=0$, there exists $1\le i\le l_1$ such that $f(\p)-\alpha_i=0$. By definition of the set $B_0$, in this case, $g(\p)-\beta_j\neq0$ for all $j\in \{1, 2, \ldots, h\}$. This means $v_\p({Q'(g)})=v_\p({\displaystyle\prod_{j=1}^h(g-\beta_j)^{q_j}})=0,$
from which it follows that
$
v_\p(\frac{Q'(g)}{G(f)})\leq0.$
Thus
\begin{align}\label{a2}\min\{v_\p(\frac{P'(f)}{G(f)}),v_\p(\frac{Q'(g)}{G(f)})\}\le 0.\end{align}
Combining (\ref{a1}) and (\ref{a2}) gives
\begin{align*}\min\{v_\p(\frac{P'(f)}{G(f)}),v_\p(\frac{Q'(g)}{G(f)})\}\le 0\end{align*}
for all points $\p\in C$ such that  $v_\p(f)\geq0$.
Together with  (\ref{h}), and the facts $n-1=\displaystyle\sum_{i=1}^l p_i,$ and $d= \displaystyle\sum_{i=l_1+1}^l p_i,$ we have
$$\displaystyle\sum_{i\in B_0}p_ih(f) \leq h({P'(f)},{Q'(g)}).$$
The above inequality and  inequality (\ref{bdtCT})  in Lemma~\ref{lmCT}~(ii) about the upper bound  for $h(P'(f),Q'(g))$
 give
$$(\displaystyle\sum_{i\in B_0}p_i-\frac{m+n}{m})h(f)\leq2\gen-2,$$
which is  the assertion (b).

For (c), let $$G_1(g):=\displaystyle\prod_{i\notin B_1}{(g-\beta_i)}^{q_i}=\displaystyle\prod_{i=h_1+1}^h{(g-\beta_i)}^{q_i},$$
and $d:=\deg G_1=\displaystyle\sum_{i=h_1+1}^h q_i.$
Similar to (b), we have
\begin{align}\label{hQ}h(P'(f),Q'(g)) &= h\left(\frac{P'(f)}{G_1(g)},\frac{Q'(g)}{G_1(g)}\right)\cr
&=-\displaystyle\sum_{v_\p(g)<0}\min\{v_\p(\frac{P'(f)}{G_1(g)}),v_\p(\frac{Q'(g)}{G_1(g)})\}\cr
&\quad\quad-\displaystyle\sum_{v_\p(g)\geq0}\min\{v_\p(\frac{P'(f)}{G_1(g)}),v_\p(\frac{Q'(g)}{G_1(g)})\}\cr
&= (m-d-\frac{m}{n})h(g)-\displaystyle\sum_{v_\p(g)\geq0}\min\{v_\p(\frac{P'(f)}{G_1(g)}),v_\p(\frac{Q'(g)}{G_1(g)})\}.
\end{align}
We still have to  prove at a point $\p\in C$ satisfying  $v_\p(g)\geq0$ that
\begin{align}\label{min}\min\{v_\p(\frac{P'(f)}{G_1(g)}),v_\p(\frac{Q'(g)}{G_1(g)})\}\leq0.\end{align}
Indeed,
we have
$$\frac{Q'(g)}{G_1(g)}=\frac{\displaystyle\prod_{i=1}^{h}{(g-\beta_i)}^{q_i}}
{\displaystyle\prod_{i=h_1+1}^{h}{(g-\beta_i)}^{q_i}}=\displaystyle\prod_{i=1}^{h_l}{(g-\beta_i)}^{q_i}.$$
Hence, if $\frac{Q'(g)}{G_1(g)}(\p)\ne0$ then $v_\p(\frac{Q'(g)}{G_1(g)})=0$, and we are done. If $\frac{Q'(g)}{G_1(g)}(\p)=0$, then there exists a $1\le j\le h_1$ such that $g(\p)-\beta_j=0$. By definition of the set $B_1$, in this case, $f(\p)-\alpha_i\neq0$ for all $i\in \{1, 2, \ldots, l\}$. This means $v_\p({P'(f)})=v_\p({\displaystyle\prod_{i=1}^l(f-\alpha_i)^{p_i}})=0,$
which implies
$
v_\p(\frac{P'(f)}{G_1(g)})\leq0.$ Therefore
$$\min\{v_\p(\frac{P'(f)}{G_1(g)}),v_\p(\frac{Q'(g)}{G_1(g)})\}\leq0$$
for all points $\p\in C$ satisfying $v_\p(g)\geq0$.
The equalities (\ref{hQ}) and (\ref{min}) imply
$$(\displaystyle\sum_{i=1}^{h_1}q_i +1-\frac mn) h(g) \leq h({P'(f)},{Q'(g)}),$$
which combines with (\ref{bdtCT}) to give
$$(\displaystyle\sum_{i\in B_1}q_i  -\frac{2m}{n}) \ h(g)\leq 2\gen-2.$$
The assertion (c) is therefore proved.
\end{proof}

\section{Proof of Theorem \ref{th2}}

\noi {\bf Notation.} \quad Recall that we have set:
\begin{align*}&A_0:=\{(i, j)\mid 1\leq i\leq l, 1\leq j\leq h, \
P(\alpha_i)=Q(\beta_{j})\},\\
&A_1:=\{(i, j)\mid (i, j)\in A_0,\  p_i>q_j\},\\
&A_2:=\{(i, j)\mid (i, j)\in A_0 ,\
p_i<q_j\},\end{align*}
  and we put $l_0:=\#A_0$.

When the polynomial $P$ and $Q$ satisfy the hypothesis I, the following lemma will  bound the cardinality of $A_0$.

\begin{lemma}[see \cite{AE}]\label{lm8} Let $P(X)$ and $Q(X)$  satisfy
Hypothesis I. Then for each $i$,
  $1\leq i\leq l$,   there exists at most one $j$, $1\leq j\leq h$,  such
that $P(\alpha_i)=Q(\beta_j)$.
  Moreover, $l_0\leq \min\{l, h\}.$\end{lemma}

\begin{proof} For each $i, (1\leq i\leq l)$, assume that there exist
$j_1, j_2,\   1\leq j_1, j_2\leq h$,  such that
$P(\alpha_i)=Q(\beta_{j_1})$ and $P(\alpha_i)=Q(\beta_{j_2})$. This
implies that $Q(\beta_{j_1})=Q(\beta_{j_2})$  and hence
$j_1=j_2$ because  $Q$ satisfies  Hypothesis I. Similarly, there
exists at most one $i, (1\leq i\leq l)$ such that
$P(\alpha_i)=Q(\beta_{j})$
  for each $j, (1\leq i\leq h)$. This ends the proof of  Lemma \ref{lm8}.
\end{proof}

By Lemma \ref{lm8},
without loss of generality we may assume from now that 
\begin{align*}&A_0=\{(1,j(1)), \ldots, (l_0,j(l_0))\};\\
&A_1=\{(1,j(1)), \ldots, (l_1,j(l_1))\}.
\end{align*}

\begin{proof}[\textbf{Proof of Theorem \ref{th2}.}] The idea to prove this theorem is similar to Theorem~\ref{th1} in that we have to find a polynomial of low degree which can cancel all the common zeros of $P'(f)$ and $Q'(g)$.
 We take
$$G_2:=\displaystyle\prod_{i=1}^{l_1}{(g-\beta_{j(i)})}^{q_{j(i)}}
\displaystyle\prod_{i=l_1+1}^{l_0}{(f-\alpha_i)}^{p_i}.$$
We have
\begin{align}\label{hPQ}h(P'(f),Q'(g)) &= h\Big(\frac{P'(f)}{G_2},\frac{Q'(g)}{G_2}\Big)\cr
&=-\displaystyle\sum_{v_\p(f)<0}\min\Big\{v_\p\Big(\frac{P'(f)}{G_2}\Big),v_\p\Big(\frac{Q'(g)}{G_2}\Big)\Big\}\cr
&\quad\quad-\displaystyle\sum_{v_\p(f)\geq0}\min\Big\{v_\p\Big(\frac{P'(f)}{G_2}\Big),v_\p\Big(\frac{Q'(g)}{G_2}\Big)\Big\}.
\end{align}
We first consider a point $\p\in C$ such that $v_\p(f)<0.$
We have $v_\p(P'(f))=(n-1)v_\p(f).$
Since $P(f)=Q(g),$ it follows that $v_\p(g)<0$ and $nv_\p(f)=mv_\p(g)$. We also have
$v_\p(Q'(g))=(m-1)v_\p(g)$. Hence,
$$v_\p(P'(f))<v_\p(Q'(g)),$$
and
\begin{align*}
v_\p(G_2)&=\Big(\displaystyle\sum_{i=1}^{l_1} q_{j(i)}\Big) \ v_\p(g)
+\Big(\displaystyle\sum_{i=l_1+1}^{l_0} p_i \Big)\ v_\p(f)\cr
&=\Big(\displaystyle\sum_{i=1}^{l_1} \frac{n}{m}q_{j(i)}
+\displaystyle\sum_{i=l_1+1}^{l_0}p_i\Big) \ v_\p(f).
\end{align*}
We remark that
$$n-1=\displaystyle\sum_{i=1}^lp_i=
\displaystyle\sum_{i=1}^{l_1}p_i
+\displaystyle\sum_{i=l_1+1}^{l_0}p_i
+\displaystyle\sum_{i=l_0+1}^lp_i.$$
Therefore
\begin{align}\label{7}
-&\displaystyle\sum_{v_\p(f)<0} \min\Big\{v_\p\Big(\frac{P'(f)}{G_2}\Big),v_\p\Big(\frac{Q'(g)}{G_2}\Big)\Big\}\cr
&=-\displaystyle\sum_{v_\p(f)<0}(\min\{v_\p(P'(f)),v_\p(Q'(g))\}-v_\p(G_2))\cr
&=-\displaystyle\sum_{v_\p(f)<0}\Big((n-1)
-(\displaystyle\sum_{i=1}^{l_1} \frac{n}{m}q_{j(i)}
+\displaystyle\sum_{i=l_1+1}^{l_0}p_i) \Big)v_\p(f)\cr
&=\Big((\displaystyle\sum_{i=1}^{l_1} p_i-\frac{n}{m}q_{j(i)})+
(\displaystyle\sum_{i=l_0+1}^{l} p_i)\Big) \ h(f).
\end{align}
Now we consider $\p\in C$ such that $v_\p(f)\geq0,$ and we will prove
$$\displaystyle\sum_{v_\p(f)\geq0}\min\Big\{v_\p\Big(\frac{P'(f)}{G_2}\Big),v_\p\Big(\frac{Q'(g)}{G_2}\Big)\Big\}\leq0.$$
Indeed, we have
\begin{align*}
\frac{P'(f)}{G_2}&=\frac{\displaystyle\prod_{i=1}^l{(f-\alpha_i)}^{p_i}}
{\displaystyle\prod_{i=1}^{l_1}{(g-\beta_{j(i)})}^{q_{j(i)}}
\prod_{i=l_1+1}^{l_0}{(f-\alpha_i)}^{p_i}}=\frac{\displaystyle\prod_{i=1}^{l_1}{(f-\alpha_i)}^{p_i}
{\displaystyle\prod_{i=l_0+1}^l{(f-\alpha_i)}^{p_i}}}{\displaystyle\prod_{i=1}^{l_1}{(g-\beta_{j(i)})}^{q_{j(i)}}},
\end{align*}
and
\begin{align*}
\frac{Q'(g)}{G_2}&=\frac{\displaystyle\prod_{j=1}^h{(g-\beta_j)}^{q_j}}
{\displaystyle\prod_{i=1}^{l_1}{(g-\beta_{j(i)})}^{q_{j(i)}}
\displaystyle\prod_{i=l_1+1}^{l_0}{(f-\alpha_i)}^{p_i}}\cr
&=\frac{\displaystyle\prod_{j\notin\{j(1),\dots,j(l_1)\}}{(g-\beta_j)}^{q_j}
\displaystyle\prod_{i=l_1+1}^{l_0}{(g-\beta_{j(i)})^{q_{j(i)}}}}{\displaystyle\prod_{i=l_1+1}^{l_0}{(f-\alpha_i)}^{p_i}}
\end{align*}

If $\frac{P'(f)}{G_2}(\p)\ne0$ then $v_{\p}\Big(\frac{P'(f)}{G_2}\Big)=0.$ Hence
$$\min\Big\{v_\p\Big(\frac{P'(f)}{G_2}\Big),v_\p\Big(\frac{Q'(g)}{G_2}\Big)\Big\}\leq0.$$

If $\frac{P'(f)}{G_2}(\p)=0$ then, since we are considering a point $\p\in C$ satisfying $v_{\p}(g)\ge 0,$ i.e. $\p$ is not a pole of $g$,  $\frac{P'(f)}{G_2}(\p)=0$ only when  the numerator is equal to zero at $\p$, in which case
there exists an $i$ with either $i\in\{1,\dots, l_1\}$ or $i\in\{l_0+1,\dots, l\}$  such that $f(\p)-\alpha_i=0.$
Suppose  that $f(\p)-\alpha_i=0$ for some $i$ such that $i\in\{1,\dots, l_1\}$.  By definition of $A_1$, we have $g(\p)-\beta_{j(i)}=0,$ and by Lemma~\ref{lm8}, for each $i$ there exits at most one $j(i)$ such that $g(\p)-\beta_{j(i)}=0$.
Looking at the ratio $\frac{Q'(g)}{G_2}$, we see that the factor of the form $g-\beta_{j(i)}$ with $(i,j(i))\in A_1$ is canceled, which means
$\frac{Q'(g)}{G_2}(\p)\ne 0.$ So, we have $v_{\p}\big(\frac{Q'(g)}{G_2}\big)= 0$. 
 Suppose that  $f(\p)-\alpha_i=0$ for some $i$ such that $i\in\{l_0+1,\dots, l\}$.   By definition of the set $A_0$, $P(\alpha_i)\ne Q(\beta_j) $ for all $1\le j\le h.$ Therefore, $g(\p)-\beta_j\ne 0$ for any $j$, which means $\frac{Q'(g)}{G_2}(\p)\ne 0,$ i.e  $v_{\p}\big(\frac{Q'(g)}{G_2}\big)= 0$.
Hence, in either case, we always have $\min\Big\{v_\p\big(\frac{P'(f)}{G_2}\big),v_\p\big(\frac{Q'(g)}{G_2}\big)\Big\}\leq0.$

Therefore, for  all of $\p\in C$ satisfying $v_{\p}(g)\ge 0,$
\begin{align}\label{8}
\min\Big\{v_\p\big(\frac{P'(f)}{G_2}\big),v_\p\big(\frac{Q'(g)}{G_2}\big)\Big\}\leq0
\end{align}
holds.

Combining (\ref{hPQ}), (\ref{7}) and (\ref{8}) gives
$$\Big(\displaystyle\sum_{i=1}^{l_1}\big(p_i-\frac{n}{m}q_{j(i)}\big)+
\displaystyle\sum_{i=l_0+1}^{l}p_i\Big) \ h(f) \leq h({P'(f)},{Q'(g)}).$$
Together with
(\ref{bdtCT}),  we have
$$\Big((\displaystyle\sum_{i=1}^{l_1}p_i-\frac{n}{m}q_j)+
(\displaystyle\sum_{i=l_0+1}^{l}p_i)-\frac{m+n}{m}\Big) \ h(f)\leq
2\gen-2.$$
The theorem is proved.
\end{proof}

\section{Proof of Theorem \ref{th3}}
In the proof of Theorem \ref{th3}, we will need the following lemmas.
\begin{lemma}\label{lm10}  Suppose there  are non-constant
functions $f$ and $g$ in $\K$ such that $P(f)=Q(g)$.
If $v_\p(f-\alpha_i)>0$ and $v_\p(g-\beta_j)>0$
at a point $\p\in C$, then
$$(p_i+1)v_\p(f -\alpha_i)=(q_j+1)v_\p(g -\beta_j).$$
\end{lemma}

\begin{proof} If  $v_\p(f-\alpha_i)>0$ and
$v_\p(g-\beta_j)>0$, then $P(\alpha_i)=Q(\beta_j)$ since $P(f)=Q(g)$.
Since $\alpha_i, \beta_j$ are zeros of $P', Q'$, with the multiplicities
$p_i, q_j$ respectively, we have the following expansions of $P$ at $\alpha_i$
and $Q$ at $\beta_j$:
\begin{align*}&P(X)=P(\alpha_i)+b_{i, p_i+1}(X-\alpha_i)^{p_i+1}+ \ldots +b_{i, n}(X-\alpha_i)^{n},\cr
&Q(X)=Q(\beta_j)+c_{j, q_j+1}(X-\beta_j)^{q_j+1}+ \ldots +c_{j, m}(X-\beta_j)^{m}.\end{align*}
We have
\begin{align*}  0&=P(f)-Q(g)\cr
&=  [b_{i, p_i+1}(f-\alpha_i)^{p_i+1}+\{\text {Higher order terms in} \ (f -\alpha_i) \}]\cr
&\qquad \ -[c_{j, q_j+1}(g-\beta_j)^{q_j+1}+\{\text{Higher order terms in} \ (g-\beta_j) \}].
\end{align*}
Therefore
$$(p_i+1)v_\p(f -\alpha_i)=(q_j+1)v_\p(g -\beta_j).$$
\end{proof}

\begin{proof}[\textbf{Proof of Theorem \ref{th3}.}] By \cite[Lemma~6]{AE}, if  $P(X)$ and $Q(X)$ satisfy one of the conditions (A), (B), (C) or (D)
 then the curve $P(X)-Q(Y)$ either has a linear factor or it is irreducible of genus 0.  If the case (E) holds, then $n=m=2$ and $P(X)-Q(Y)=(X-\alpha_1)^2-(Y-\beta_1)^2+c$
for some constant $c$. Hence either the curve $P(X)-Q(Y)$ has a linear factor or it is irreducible of genus 0.
Therefore,  for all of these cases, there exist two non-constant rational functions $f$ and $g$ in
$\K$ such that $P(f)=Q(g)$.
 The necessary condition is proved.

We now turn to the proof of the sufficient condition of the theorem.
Suppose that $P(X)$ and $Q(X)$ satisfy none of the conditions (A), (B), (C), (D) or (E).
Assume that
$f$ and $g$ are two non-constant rational functions in $\K$ such that $P(f)=Q(g)$.
 When the polynomials $P$ and $Q$ satisfy Hypothesis I, by Lemma \ref{lm8},
without loss of generality we may assume that $A_0$ is of the form
$\{(1,j(1)), \ldots, (l_0,j(l_0))\}$ such that the $p_i$ are non-increasing, it means
$p_1\geq p_2\geq \ldots \geq p_{l_0}.$ 

By Theorem \ref{th2}, together with the hypothesis $\gen=0, n=m$,  the
right-hand sides of the inequalities in Theorem \ref{th2} are negative, therefore
\begin{align}\label{14}&(\sum_{(i,j(i))\in A_1}p_i-q_{j(i)})+(\sum_{i=l_0+1}^l p_i)-2<0, \end{align}
 \text{which also implies}
\begin{align}\label{14.1}(\sum_{(i,j)\in A_2}q_{j(i)}-p_i)+(\sum_{j\notin\{ {j(1)},\dots, {j(l_0)}\}}q_j)-2<0.
 \end{align}
From the inequality (\ref{14}) and (\ref{14.1}), we have 
\begin{align}\tag{${\cal H}$}  & p_i= 1 \text{ for  all }
i\ge l_0+1;\  q_j= 1 \text{ for  all } j\notin\{ {j(1)},\dots, {j(l_0)},\  |p_i-q_{j(i)}|\le 1  \text{ for  } i\le l_0;\\
&|l-h|\le 1 \text{  and  } l_0\ge \max\{l, h\}-1.\nonumber
\end{align}
We will consider the following cases.

\medskip
  \noindent{\it Case 1.}\quad $l_0=0.$
\smallskip

By the statement ($\mathcal{H}$), we have $\max\{l, h\}\le l_0 +1=1.$

If $l=0$, then $P(X)$ is of the form $uX+v$ with $uv\neq0$, i.e $n=1.$
Since $n=m$, it follows that $m=1$ and $Q(Y)$ is of the form $sY+t$ with $st\neq0.$
Therefore, $P(X)-Q(Y)$ has a linear factor. This is the exceptional case corresponding to
condition (A), which is excluded.

If $l=1$, then, by the hypothesis
$n=m$ and the condition ($\mathcal{H}$),  $h=l=1$ and $q_1=p_1=1$. This is the exceptional case (E).

\medskip
  \noindent{\it Case 2.}\quad $l_0=1.$
\smallskip

Then, by the statement ($\mathcal{H}$), we have $\max\{l, h\}\le l_0 +1=2.$

Suppose first that $l=1$. By the conditions $n=m$ and  ($\mathcal{H}$), the case  $p_1<q_{j(1)}$ cannot happen.
 We only have to consider the following possibilities.
If $p_1=q_{j(1)}$ then, since $n=m$ and the statement ($\mathcal{H}$), we have $h=1$. Since $l_0=1$ we have $P(\alpha_1)=Q(\beta_{j(1)})$, therefore
$P(X)-Q(Y)$ is of the form $u(X-\alpha_1)^{p_1+1}-v(Y-\beta_{j(1)})^{p_1+1}$
with $uv\neq0$. Therefore, $P(X)-Q(Y)$ has a linear factor. This is the exceptional case (A).
If $p_1>q_{j(1)}$ then $p_1=q_{j(1)}+1, h=2$ and $q_2=1$. 
This is the exceptional case (B).

Suppose that $l=2$. Then $p_2=1$ because of the condition $p_i\leq1$ for $i\geq l_0+1$
in the statement ($\mathcal{H}$). On the other hand, the case  $p_1>q_{j(1)}$ cannot hold. Therefore, we consider
the following possibilities.
If $p_1=q_{j(1)}$ then $h=2$ and $q_2=1.$  This is the exceptional case (C).
If $p_1<q_{j(1)}$ then $q_{j(1)}=p_1+1$ and $h=1$.  This is the exceptional case (B).

\medskip
  \noindent{\it Case 3.}\quad $l_0\geq2.$
\smallskip

For each $(i_1, j(i_1))$ and $(i_2, j(i_2))$ in $A_0$, we define $L_{i_1, i_2}(f, g)=L_{i_1,i_2}$
as follows
\begin{align}\label{16}
 L_{i_1,i_2} :=(g-\beta_{j(i_1)})-\frac{\beta_{j(i_1)}-\beta_{j(i_2)}}{\alpha_{i_1}-\alpha_{i_2}}(f-\alpha_{i_1}),
\end{align}
  which can also be expressed as
\begin{align}\label{17}
L_{i_1,i_2} :=(g-\beta_{j(i_2)})-\frac{\beta_{j(i_1)}-\beta_{j(i_2)}}{\alpha_{i_1}-\alpha_{i_2}}(f-\alpha_{i_2}).
\end{align}

Now we take
$$\displaystyle G := L_{1,2}^{p_1+p_2-2+{\sum_{i=l_0+1}^l}p_i}\prod_{i=3}^{l_0}(f-\alpha_i)^{p_i}.$$
We have
\begin{align*}h(P'(f),Q'(g)) &= h\Big(\frac{P'(f)}{G},\frac{Q'(g)}{G}\Big)\cr
&=-\displaystyle\sum_{v_\p(f)<0}\min\Big\{v_\p\Big(\frac{P'(f)}{G}\Big),v_\p\Big(\frac{Q'(g)}{G}\Big)\Big\}\cr
&\quad\quad-\displaystyle\sum_{v_\p(f)\geq0}\min\Big\{v_\p\Big(\frac{P'(f)}{G}\Big),v_\p\Big(\frac{Q'(g)}{G}\Big)\Big\}.
\end{align*}
We first consider a point $\p\in C$ such that $v_\p(f)<0.$
We have $v_\p(P'(f))=(n-1)v_\p(f).$
Since $P(f)=Q(g)$ and $n=m$ by hypothesis, we have $v_\p(f)=v_\p(g)=v_\p(f-\alpha_1)=v_\p(g-\beta_{j(1)})$
and $v_\p(g)<0$. Thus $v_\p(L_{1,2})\ge v_\p(f)$, therefore
\begin{align*}
v_\p(G)&\ge\Big(\displaystyle p_1+p_2-2+\sum_{i=l_0+1}^l p_i+\sum_{i=3}^{l_0}p_i\Big) \ v_\p(f)\cr
&\ge(n-1-2) \ v_\p(f)=(n-3) \ v_\p(f).
\end{align*}
Therefore, 
\begin{align}\label{19}
&-\displaystyle\sum_{v_\p(f)<0} \min  \Big\{v_\p\Big(\frac{P'(f)}{G}\Big),v_\p\Big(\frac{Q'(g)}{G}\Big)\Big\}\cr
&=-\displaystyle\sum_{v_\p(f)<0}\Big(\min\{v_\p(P'(f)),v_\p(Q'(g))\}-v_\p(G)\Big)\ge 2  h(f).
\end{align}
Together with  Lemma~\ref{lmCT}(ii) and $\gen=0$, we have
\begin{align*} 2  h(f)-\displaystyle\sum_{v_\p(f)\geq0}\min\Big\{v_\p\Big(\frac{P'(f)}{G}\Big),v_\p\Big(\frac{Q'(g)}{G}\Big)\Big\}&\le
h({P'(f)},{Q'(g)})\\
&\leq 2 h(f)-2 -
\displaystyle\sum_{\p\in C}  \min\{v_{\p}^0(d_\p f), v_{\p}^0(d_\p g)\}\\
&\leq 2 h(f)-2 -
\displaystyle\sum_{v_\p(f)\geq0}  \min\{v_{\p}^0(d_\p f), v_{\p}^0(d_\p g)\},
\end{align*}
since if $v_\p(f)<0$ then $v_{\p}^0(d_\p f)=v_{\p}^0(d_\p g)=0$. Therefore, if we can prove
\begin{align}\label{contract} \displaystyle\sum_{v_\p(f)\geq0}\min\Big\{v_\p\Big(\frac{P'(f)}{G}\Big),v_\p\Big(\frac{Q'(g)}{G}\Big)\Big\}&<
2+
\displaystyle\sum_{v_\p(f)\geq0}  \min\{v_{\p}^0(d_\p f), v_{\p}^0(d_\p g)\}
\end{align}
then we can get a contradiction. \\
Let $\p\in C$ such that $v_\p(f)\geq0.$ For our purpose, we only have to consider those $\p\in C$ such that $v_\p\Big(\frac{P'(f)}{G}\Big)>0$ and also $v_\p\Big(\frac{Q'(g)}{G}\Big)>0.$
Since $v_\p(f)\geq0$ and
\begin{align*}
\frac{P'(f)}{G}
&=\frac{\displaystyle (f-\alpha_1)^{p_1}(f-\alpha_2)^{p_2}\prod_{i=l_0+1}^l{(f-\alpha_i)}^{p_i}}
{\displaystyle L_{1,2}^{p_1+p_2-2+{\sum_{i=l_0+1}^l}p_i}},
\end{align*}
we have
$v_\p\Big(\frac{P'(f)}{G}\Big)>0$ when $(f-\alpha_i)(\p)=0$ for $i$ is one  index in the set $\{1,  2, l_0+1,\dots,l\}$. However, if $i\in\{ l_0+1,\dots,l\}$ then, by the definition of the set $A_0$, $(g-\beta_j)(\p)\ne 0$ for any $j=1,...,h$, which means 
$$v_\p\Big(\frac{Q'(g)}{G}\Big)=v_\p\Big(\frac{\displaystyle\prod_{j=1}^h{(g-\beta_j)}^{q_j}}
{\displaystyle L_{1,2}^{p_1+p_2-2+{\sum_{i=l_0+1}^l}p_i}\prod_{i=3}^{l_0}(f-\alpha_i)^{p_i}}\Big)\le 0.$$
Therefore, we only have to check at points $\p\in C$ such that
 $v_\p(f-\alpha_i)>0$ and $v_\p(g-\beta_{j(i)})>0$ for $i=1$ or 2. 

We first consider $i=1$.
 By Lemma \ref{lm10},
\begin{align}\label{20}
(p_1+1)v_\p(f-\alpha_1)=(q_{j(1)}+1)v_\p(g-\beta_{j(1)}).
\end{align}
If $p_1\geq q_{j(1)}$ then $v_\p(f-\alpha_1)\leq v_\p(g-\beta_{j(1)})$ and hence
$v_\p(L_{1,2})\ge \min\{ v_{\p} (f-\alpha_{1}), v_{\p}
(g-\beta_{j(1)}\}=v_\p(f-\alpha_1)$. Therefore
$$v_\p \Big(\frac{P'(f)}{G}\Big)\le -\Big(p_2-2+\sum_{i=l_0+1}^lp_i\Big)v_\p(f-\alpha_1),$$
which is not positive if  $p_2\geq 2$ or if we have both $p_2=1$ and $l_0+1\leq l.$ Hence
$$\min\Big\{v_\p(\frac{P'(f)}{G}),v_\p(\frac{Q'(g)}{G})\Big\}\leq0<2+
\displaystyle\sum_{v_\p(f)\geq0}  \min\{v_{\p}^0(d_\p f), v_{\p}^0(d_\p g)\},$$
which means the inequality~(\ref{contract}) holds, so we can get a contradiction for this case. 
If $p_1<q_{j(1)}$ then, by \dkh, $q_{j(1)}=p_1+1.$ Therefore,
from the equality (\ref{20}) we have
$$v_\p(f-\alpha_1)=\frac{p_1+2}{p_1+1} \ v_\p(g-\beta_{j(1)}).$$
Hence $v_\p(f-\alpha_1)> v_\p(g-\beta_{j(1)})\geq p_1+1,$ the last inequality follows from the fact that $p_1+1$ and $p_1+2$ are coprime.
 So $v_\p(L_{1,2})= v_\p(g-\beta_{j(1)}).$
Therefore
\begin{align*}
v_\p \Big(\frac{P'(f)}{G}\Big)
&\le\Big(\frac{p_1(p_1+2)}{p_1+1}-p_1-p_2+2-\sum_{i=l_0+1}^lp_i\Big)v_\p(g-\beta_{j(1)})\cr
&\leq v_\p(g-\beta_{j(1)})-\frac{1}{p_1+1}v_\p(g-\beta_{j(1)}),
\end{align*}
if either $p_2\geq 2$ or  $p_2=1$ and $l_0+1\leq l.$ 
Since $v_\p(g-\beta_{j(1)})\geq p_1+1$ , it follows that $v_\p \Big(\frac{P'(f)}{G}\Big)\leq v_\p(g-\beta_{j(1)})-1=v_{\p}(d_\p g)$. Therefore, 
\begin{align*}\min\Big\{v_\p\Big(\frac{P'(f)}{G}\Big),v_\p\Big(\frac{Q'(g)}{G}\Big)\Big\}&\le v_\p \Big(\frac{P'(f)}{G}\Big)\leq v_{\p}(d_\p g)=\displaystyle\sum_{v_\p(f)\geq0}  \min\{v_{\p}^0(d_\p f), v_{\p}^0(d_\p g)\},
\end{align*}
which shows that (\ref{contract}) holds.

Using the same arguments for the case $i=2$, we also have a contradiction if either $p_2\geq 2$ or  $p_2=1$ and $l_0+1\leq l.$  Thus, $f$ and $g$
must be constants.

The remaining case is when $p_2=1$ and $l_0=l.$ 

\medskip
\noi{\it Subcase 1.}\quad  $p_2=1$ and $l_0=l=2$.
\smallskip

Then $P'(X)$ is of the form $(X-\alpha_1)^{p_1}(X-\alpha_2)$ and $n=p_1+2$.
We have $h\geq l_0=2.$ Since $n=m$ and by \dkh, it follows that $Q'(Y)$ is only
one of the following forms
\begin{align*}
&(Y-\beta_{j(1)})^{p_1}(Y-\beta_{j(2)}), \ {\rm i.e} \quad h=2, q_{j(1)}=p_1, q_{j(2)}=1; \cr
&(Y-\beta_{j(1)})^{p_1-1}(Y-\beta_{j(2)})^2, \ {\rm i.e} \quad h=2, q_{j(1)}=p_1-1, q_{j(2)}=2; \cr
&(Y-\beta_{j(1)})^{p_1-1}(Y-\beta_{j(2)})(Y-\beta_3), \ {\rm i.e} \quad h=3, q_{j(1)}=p_1-1, q_{j(2)}=1,q_3=1.
\end{align*}
The first form corresponds to condition (C), which is excluded.
For two remaining cases, we take
$$G:=L_{1,2}(g-\beta_{j(1)})^{p_1-2}.$$
By an argument analogous to the previous, we get $f$ and $g$ are constants.

\medskip
\noi{\it Subcase 2.}\quad  $p_2=1$ and $l_0=l\geq 3$.
\smallskip

If $p_1=1$ and $l_0=l=3$ then, since $p_1\geq p_2\geq p_3$, it follows that
$P'(X)$ is of the form $(X-\alpha_1)(X-\alpha_2)(X-\alpha_3)$ and $n=4.$ 
Since $n=m$ and $l_0=l=3$, hence $h=3$ and $q_1=q_2=q_3=1$. Therefore, 
$Q'(Y)$ is of the form $(Y-\beta_{j(1)})(Y-\beta_{j(2)})(Y-\beta_{j(3)})$ and 
$P(\alpha_i)=Q(\beta_{j(i)}) \text{\;for any\;} i=1, 2, 3.$ This is the exceptional
case (D).

If $p_1=1$ and $l_0=l\geq 4$, then we take
$$G:=L_{1,2}L_{3,4}\prod_{i=4}^{l_0}(f-\alpha_i)^{p_i}.$$

If $p_1\geq 2$, then we take
$$G:=L_{1,2}^{p_1-1}L_{1,3}\prod_{i=4}^{l_0}(f-\alpha_i)^{p_i}.$$
For the above two cases,
by repeating arguments similarly to the Case 1, we also
get $f$ and $g$ are constants. Theorem \ref{th3} is therefore proved. 
\end{proof}

\end{document}